\documentclass[12pt]{article}
\usepackage[left = 1.4cm, right = 1.8cm, top = 1.5cm, bottom = 2cm]{geometry}

\usepackage{amssymb}
\usepackage{amsmath}
\usepackage{adjustbox}
\usepackage{multirow}
\usepackage{enumerate}
\usepackage{enumitem}
\usepackage[colorlinks,
            linkcolor=blue,
            anchorcolor=blue,
            citecolor=blue
            ]{hyperref}
\usepackage{amsthm}
\usepackage{authblk} 
\usepackage{booktabs}
\usepackage{caption}
\usepackage{arydshln}

\makeatletter
\newcommand{\rmnum}[1]{\romannumeral #1}
\newcommand{\Rmnum}[1]{\expandafter\@slowromancap\romannumeral #1@}
\makeatother

\newtheorem{theorem}{Theorem}
\newtheorem{lemma}{Lemma}
\newtheorem{proposition}{Proposition}

\newtheorem{assumption}{Assumption}

\newtheorem{remark}{Remark}

\begin{document}

\title{Regularity and high-order time stepping for semilinear subdiffusion equations with singular initial data beyond the $L^\infty$ framework}

\author[1]{Runjie Zhang\thanks{Email: rjzhangxtu@163.com}}
\author[1]{Dongling Wang\thanks{Corresponding author: Wang. Email: wdymath@xtu.edu.cn. The research of Dongling Wang is supported in part by the National Natural Science Foundation of China under Grants 12271463 and the Natural Science Foundation of Hunan Province under Grant 2026JJ50363.}}

\affil[1]{Hunan Research Center of the Basic Discipline Fundamental Algorithmic Theory and Novel Computational Methods, School of Mathematics and Computational Science, Xiangtan University, Xiangtan, Hunan 411105, China}

\maketitle

\abstract{This paper aims to analyze a numerical scheme for semilinear subdiffusion problems with singular initial data beyond the $L^\infty$ framework. The main difficulty lies in the stronger singular behavior of the nonlinear term compared with previous analyses. Since the singular initial datum is too rough to guarantee a uniform $L^\infty$ bound for the solution, the usual Lipschitz framework in the base space is no longer sufficient. The analysis must instead be carried out in weaker fractional Sobolev-type spaces, where nonlinear composition is more delicate and the term $f(u(t))$ may exhibit an amplified singularity relative to that of $u(t)$. To overcome this difficulty, we exploit the smoothing properties of the subdiffusion solution operators and formulate suitable nonlinear assumptions in fractional operator spaces. These smoothing estimates allow part of the singularity to be transferred from the nonlinear term to the solution operators, where it can be controlled. Under these assumptions, we establish well-posedness and regularity results for the mild solution and derive a pointwise-in-time error estimate for the exponential convolution quadrature method. Numerical experiments confirm the predicted convergence rates. }

\newcommand{\keywords}[1]{\textbf{Keywords:} #1}


\newcommand{\pacs}[1]{\textbf{Mathematics Subject Classification (2010):} #1}

\section{Introduction}
This paper analyzes high-order time-stepping methods for the semilinear subdiffusion problem
\begin{equation}\label{eq_main}
    \begin{cases}
    \partial_t^\alpha u(t) + Au(t) = f(u(t)),\quad 0<t\leq T,
    \\
    u(0)=u_0,
\end{cases}
\end{equation}
on the Hilbert space $X = L^2(\Omega)$, where $\Omega$ is a convex polygonal domain in $\mathbb{R}^d$, $d=1,2,3$. 
The operator $\partial_t^\alpha u$ denotes the Caputo fractional time derivative of order $\alpha\in(0,1)$ in time, defined by
\begin{equation}\label{def_fracDeri}
    \partial_t^\alpha u:=\frac1{\Gamma(1-\alpha)}\int_0^t(t-s)^{-\alpha}\frac{\mathrm{d}}{\mathrm{d}s}u(s)\:\mathrm{d}s,
\end{equation}  
with $\Gamma(z)=\int_0^\infty s^{z-1}e^{-s}\text{d}s$ being the gamma function. 
The spatial operator $A:D(A)\subset X\to X$ is the sectorial operator
induced by a second-order elliptic operator with suitable boundary conditions. The nonlinear term $f$ is assumed to be a smooth function on $\mathbb{R}$. The initial value is allowed to be nonsmooth: $u_0\in D(A^\gamma)$ for some $0<\gamma\leq d/4$. A function $u \in \mathcal{C}([0, T]; L^2(\Omega))$ is called a mild solution of \eqref{eq_main} if it satisfies the integral equation
\begin{equation}\label{eq_mildSolu}
    u(t)=F(t)u_0+\int_0^t E(t-s)f(u(s))\text{d}s,
\end{equation}
where $F(t)$ and $E(t)$ are the solution operators defined in \eqref{def_solOpera} below. The mild and strong formulations are equivalent whenever the solution possesses sufficient regularity. 

The subdiffusion equation in the form \eqref{eq_main} has attracted considerable attention in the development of stable and accurate numerical methods, along with rigorous numerical analysis, owing to its remarkable ability to model a wide range of anomalously slow transport processes. 
Two major strategies are available for approximating the fractional derivative \eqref{def_fracDeri}: convolution quadrature (CQ) \cite{lubich86,lubich88,cuesta06,yuste06} and L1-type methods \cite{langlands05,lin07,sun06,alikhanov15}. In the stability analysis of L1-type methods, deriving explicit bounds for discrete Gronwall inequalities is of great significance, and several such inequalities have been established in the literature \cite{liao18,liao19,feng24}. For a comprehensive treatment of these two strategies, see the recent monograph \cite{JZ23}.

Solutions of subdiffusion equations typically exhibit a weak singularity at $t=0$, 
which is inherited by the nonlinear term $f(u(t))$ and can be amplified in certain norms; 
this behavior constitutes a key difficulty in the analysis of nonlinear subdiffusion problems. Nevertheless, the semilinear subdiffusion model \eqref{eq_main} has been extensively studied from both theoretical and numerical perspectives in recent years.

The first rigorous analysis, given in \cite{JL18}, proposed a general framework for the mathematical and numerical analysis of the semilinear subdiffusion problem with initial data $u_0\in H^2\cap H^1_0$. It studied a time-stepping scheme based on the backward Euler CQ scheme or the L1 method, and established an $O(\tau^\alpha)$ error bound uniform over the fixed time interval $[0,T]$. 
Subsequent work in \cite{AK19} extended the analysis to nonsmooth initial data $u_0\in \dot{H}^s$, $s\in[0,2)$, and obtained a pointwise-in-time convergence rate $O(\tau)$ for the backward Euler CQ scheme. To achieve higher convergence orders, special treatment at the initial time steps can be applied. In \cite{LW19, Kop25}, the L1 and L2 schemes on graded meshes were analyzed, and optimal convergence rates of $O(\tau^{2-\alpha})$ and $O(\tau^{3-\alpha})$ were derived, respectively. Correction of high-order BDF CQ schemes was examined in \cite{WZ20}, yielding a convergence rate of at most $1+2\alpha$ for general initial data in $H_0^1 \cap \mathcal{C}^2(\Omega)$. More recently, \cite{ECQ22, ESM23} introduced a new class of high-order methods that are easy to construct, with the latter  achieving spectral convergence. Similar to CQ methods, these approaches start from the Laplace transform representation \eqref{eq_mildSolu} of the solution, but they employ a quadrature approximation of contour integrals, exponential integrators for the ODE arising from the nonlinear term, and locally refined stepsizes to resolve the initial singularity. 
Other aspects of nonlinear subdiffusion equations have also been investigated. The analysis of the L1 scheme on graded meshes was carried out in \cite{HS20}, where the optimal convergence rate in the $H^1$ norm was established. Convolution quadrature was applied to a quasilinear subdiffusion model in \cite{lopez25, Jin25}, and the corresponding solution regularity theory and error bounds were developed. It was proved in \cite{wang23} that, under suitable structural assumptions, the discrete solutions to time-fractional evolutionary equations exhibit the long-time power-law decay rate. In addition, a numerical scheme for the backward problem of semilinear subdiffusion equations was proposed and analyzed in \cite{WYZ25}.

However, prior analyses of semilinear subdiffusion equations typically require the nonlinearity $f$ to satisfy a Lipschitz condition in the base space. 
This requirement is justified when the solution remains uniformly bounded in $L^\infty(\Omega)$ over time (see \cite{JL18,ECQ22}), because $L^\infty$ is closed under pointwise multiplication and composition with smooth functions; i.e.,
$$
    \Vert v_1 v_2\Vert_{L^\infty}\leq \Vert v_1\Vert_{L^\infty} \Vert v_2\Vert_{L^\infty} \text{ and } f(v_1)\in L^\infty \quad \text{for } v_1, v_2\in L^\infty.$$  
Nevertheless, in practice, some initial data may belong to $L^2$ but not to $L^\infty$ \cite{ghergu23, LiMa22NS, brezis96} --- for instance, in biological models of species invasion or local outbreaks the initial density can be extremely concentrated, such as $|x|^{-1/4}$. In these cases the $L^\infty$-norm of the solution is singular near $t=0$, rendering the global Lipschitz assumption invalid and calling for new nonlinear constraints. Hence there is a gap for semilinear subdiffusion problems when the initial datum lies strictly between $L^2$ and $L^\infty$. 

Fractional Sobolev spaces $H^{2\gamma}$ ($0<\gamma\le d/4$) naturally fill this regularity window and serve as our starting point. Unlike $L^\infty$, these spaces are less well-behaved: if the solution $u(t)$ is merely bounded in such a space, the nonlinear term $f(u(t))$ can possess a stronger singularity, making the numerical analysis substantially harder. One therefore needs to quantify and then mitigate this singularity.

For semilinear parabolic equations, analogous difficulties have been overcome. By considering polynomial nonlinearities, the authors of \cite{ERK_FNM, HERK25} exploit pointwise multiplication properties in fractional Sobolev spaces, together with duality arguments, to formulate new nonlinear constraints. Under these constraints, they derived sharp error estimates for exponential 
Runge--Kutta methods with nonsmooth initial data. 
In the present work, we distill these key nonlinear constraints into simplified Assumptions \ref{ass:Lips_L2}--\ref{ass:FrechetDiff}, and apply them to the numerical analysis of fractional subdiffusion equations.

A key idea in \cite{ERK_FNM, HERK25} is the factorization $I = A^{\theta}A^{-\theta}$, 
which shifts part of the singularity away from the nonlinearity $f(u)$ and lets it be absorbed by the parabolic semigroup $e^{-tA}$. 
The subdiffusion solution operator $E(t)$ possesses analogous smoothing properties. Moreover, the method proposed in \cite{ECQ22} shares the philosophy of the exponential integrators above: it starts from the mild solution, approximates the solution operator via a high-accuracy algorithm (contour integral), and treats the nonlinear part by polynomial extrapolation. Furthermore, this method is efficient---it is high-order, parallelizable, and avoids the large memory consumption typical of time-fractional PDEs. We therefore adopt this method as our discretization and carry out its error analysis under the new assumptions.

The main contributions of this work are outlined as follows:
\begin{enumerate}
    \item We establish a solution theory for problem \eqref{eq_main} with singular initial data. More precisely, we adapt the semilinear parabolic argument of \cite[Theorem 4.1]{ABsPEE} to the fractional-in-time setting by replacing the analytic semigroup estimates with the smoothing estimates for the subdiffusion solution operators in Lemma \ref{le:smoo_solOpe}. This yields the local well-posedness result in Theorem \ref{thm:localWelPos}. 
    We also establish positive-time regularity estimates for the solution and the nonlinear term in Proposition \ref{pro:solRegu}, extending the $L^\infty$-based estimates of \cite[Theorem 3.2]{ECQ22} to suitable negative-order fractional domain spaces and yielding milder time singularities.

    \item  We prove a pointwise-in-time error estimate for the high-order time-stepping schemes of \cite{ECQ22}; see Theorem~\ref{thm:convergence}. The result extends \cite[Theorem 2.1]{ECQ22} by allowing initial data outside $L^\infty$ and by relaxing the requirement on the exponent in the power-law temporal mesh \eqref{def_gradeMesh}. The proof follows a broadly similar strategy to that in \cite{ECQ22}, but is substantially more delicate as it is carried out in fractional domain spaces and must additionally handle the singularity transferred from the nonlinear term.
\end{enumerate}

The paper is structured as follows. Section \ref{sec:NumerDiscre} introduces the notational conventions and presents the construction of the time-stepping method. Section \ref{sec:assAres} states the principal hypotheses and conclusions. The proofs of these results are given in Section \ref{sec:proof_res}. Numerical examples are presented in Section \ref{sec:num_exper} to validate the theoretical findings.

\section{Numerical Discretization}\label{sec:NumerDiscre}

We start by introducing some notation. We take the underlying space $X = L^2$ equipped with the norm $\Vert \cdot \Vert$. Moreover, when no ambiguity arises, the norm on the space $\mathcal{L}(X)$ of bounded linear operators from $X$ to $X$ is also denoted by $\Vert \cdot \Vert$. 
Let $A$ be a densely defined closed linear operator on $X$ which is generated by a second-order elliptic operator subject to various boundary conditions (see, e.g., \cite[Chapter 2]{ABsPEE}) and satisfies the resolvent estimate
$$
    \Vert (z + A)^{-1} v\Vert \leq C_\varphi |z|^{-1} \Vert v\Vert \quad \forall z \in \Sigma_\varphi,\: \varphi \in (\pi/2, \pi),\: \forall v \in L^2(\Omega),
$$
where $\Sigma_\varphi = \{z \in \mathbb{C} : |\arg(z)| \leq \varphi\}$ is a sector in the complex plane. 
For any $s \in \mathbb{R}$, we denote by $D(A^s)$ the Banach space equipped with the graph norm $\Vert\cdot\Vert_s=\Vert A^s\cdot\Vert$. Throughout this paper, $C$ stands for a generic positive constant and $\varphi(\cdot)$ denotes an increasing continuous function; both may vary across instances. We also write $a\simeq b$ for two nonnegative quantities $a$ and $b$ if $C^{-1}b\leq a\leq Cb$, meaning that $a$ and $b$ are comparable. 

The solution operators $F(t)$ and $E(t)$ associated with the subdiffusion equation are defined via the inverse Laplace transform of the operators $z^{\alpha-1}(z^\alpha+A)^{-1}$ and $(z^\alpha+A)^{-1}$, respectively; i.e.,
\begin{equation}\label{def_solOpera}
\begin{aligned}
    F(t)&=\frac{1}{2\pi i}\int_{\text{Re}(z)=\sigma}e^{zt} z^{\alpha-1}(z^\alpha+A)^{-1}\mathrm dz,\\E(t)&=\frac{1}{2\pi i}\int_{\text{Re}(z)=\sigma} e^{zt}(z^\alpha+A)^{-1}\mathrm dz.
\end{aligned}
\end{equation}
The integration contours in the definitions can be deformed, respectively, into the following two contours in the complex plane:
\begin{equation}\label{def_contour}
    \Gamma_\lambda=\{\lambda(1-\sin(\beta+is)):s\in\mathbb{R}\}\quad\mathrm{and}\quad\Gamma_{\tilde{\lambda}}=\{\tilde{\lambda}(1-\sin(\beta+is)):s\in\mathbb{R}\},
\end{equation}
where $\beta\in(0,\varphi-\frac\pi2)$, $\varphi\in(\frac\pi2,\pi)$, and $\lambda,\tilde{\lambda}$ are fixed positive constants. These contours lie within the region between the two sectors $\Sigma_\varphi$ and $\lambda+\Sigma_{\beta+\frac\pi2}$; see \cite[Figure 2.1]{ECQ22}.

We now proceed to the numerical discretization, which follows the scheme proposed in \cite{ECQ22}. Specifically, the mild solution \eqref{eq_mildSolu} can be rewritten as
$$\begin{aligned}
    u(t)=&\frac1{2\pi i}\int_{\Gamma_\lambda}e^{zt} z^{\alpha-1}(z^\alpha+A)^{-1} u_0 \mathrm{d}z
    \\
    &\quad+\frac1{2\pi i}\int_{\Gamma_{\tilde{\lambda}}}(z^\alpha+A)^{-1}\int_0^t e^{z(t-s)}f(u(s)) \mathrm{d}s \mathrm{d}z.
\end{aligned}$$
It is well established that the two contour integrals above can be approximated by a quadrature rule that converges spectrally as the number of quadrature points increases; i.e.,
\begin{equation}\label{eq_quadAppro}
\begin{aligned}
    u(t)&=\sum_{j=-M}^{M}\omega_je^{z_j t}z_j^{\alpha-1}(z_j^\alpha+A)^{-1}u_0+\mathcal{E}_{1,q}(t)
    \\
    &\quad+\sum_{j=-M}^{M}\tilde{\omega}_j(\tilde{z}_j^\alpha+A)^{-1} y(\tilde{z}_j,t)+\mathcal{E}_{2,q}(t),
\end{aligned}   
\end{equation}
where a total of $2M+1$ quadrature points are employed, $(\omega_j,z_j)$ and $(\tilde{\omega}_j,\tilde{z}_j)$ are two pairs of quadrature weights and nodes, $\mathcal{E}_{1,q}(t)$ and $\mathcal{E}_{2,q}(t)$ denote the corresponding quadrature errors, and $y(z, t)=\int_0^t e^{z(t- s)}f(u(s))\text{d}s$. The explicit formulas for the quadrature weights and nodes are given in \cite[Section 2]{ECQ22}.

From the definition of $y(z,t)$, one readily verifies that it satisfies the ordinary differential equation $\partial_{t} y(z, t) = z y(z, t) + f(u(t))$. 
An exponential multistep method is employed to solve this equation. 
We first introduce a temporal mesh. Because the solution exhibits a singularity at the initial time, a locally refined mesh is employed. Let $k$ be an arbitrary positive integer and let $\nu \in\big(\max(1-\frac{1+\alpha\gamma}{k},0),1\big)$ be a given constant, where 
$\gamma$ characterizes the regularity of the initial data. Consider a partition $0=t_0<t_1<\cdots<t_N=T$ of the interval $[0,T]$ with step sizes
\begin{equation}\label{def_gradeMesh}
    \tau_1 \simeq T\left(\frac\tau T\right)^{\frac1{1-\nu}}\quad\mathrm{and}\quad\tau_n=t_n-t_{n-1}\simeq\left(\frac{t_{n-1}}T\right)^\nu\tau\quad\text{for }n\geq 2,
\end{equation}
where $\tau$ is the maximum step size. The total number of time levels is $O(T/\tau)$. 

By Duhamel's principle, the function $y(z,t)$ can be expressed as
$$
    y(z, t_{n}) = e^{z \tau_{n}} y(z, t_{n-1}) + \int_{t_{n-1}}^{t_{n}} e^{z(t_{n}-s)} f(u(s)) \text{d}s.
$$
In this expression we then approximate $f(u(s))$ using a $k$-step extrapolation
\begin{equation}\label{def_extrap}
\begin{aligned}
     f(u(s)) &= \sum_{i=1}^{k} L_{ni}(s) f(u(t_{n-i})) + \mathcal{E}_f^n(s)
    \\
    &=: I_\tau f(u)(s) + \mathcal{E}_f^n(s) \quad \text{for } s \in (t_{n-1}, t_n] \text{ and } n \geq k,
\end{aligned}
\end{equation}
where $L_{ni}(s)$ denotes the unique polynomial of degree $k-1$ satisfying 
$$L_{ni}(t_{n-j}) = \delta_{ij} \quad \text{for } j = 1, 2, \dots, k.$$

In summary, the numerical scheme is given by
\begin{equation}\label{NumSch_un}
    u_{n}=\sum_{j=-M}^{M}w_{j}(t_{n})e^{z_{j}(t_{n})t_{n}}z_{j}^{\alpha-1}(t_{n})(z_{j}^{\alpha}(t_{n})+A)^{-1}u_{0}+\sum_{j=-M}^{M}\tilde{w}_{j}(\tilde{z}_{j}^{\alpha}+A)^{-1}y_{n}(\tilde{z}_{j}),
\end{equation}
with $y_n(\tilde{z}_j)$ defined as
\begin{equation}\label{NumSch_yn}
    y_n(\tilde{z}_j)=\left\{ 
    \begin{aligned}
        &e^{\tilde{z}_j\tau_n}y_{n-1}(\tilde{z}_j) + \int_{t_{n-1}}^{t_n} e^{\tilde{z}_j(t_n-s)}f(u_{n-1})\text{d}s &&\text{for } n=1,...,k-1, 
        \\
        &e^{\tilde{z}_j\tau_n}y_{n-1}(\tilde{z}_j)+\sum_{i=1}^k\bigg(\int_{t_{n-1}}^{t_n}e^{\tilde{z}_j(t_n-s)}L_{ni}(s)\mathrm{d}s\bigg)f(u_{n-i})  &&\text{for } n\geq k.
    \end{aligned}\right.
\end{equation}

\section{Nonlinear assumptions and main results}\label{sec:assAres}

In this section we present the principal hypotheses and the main conclusions of the paper.  
We first examine the well-posedness of equation \eqref{eq_main}. Following \cite[Chapter 4]{ABsPEE}, we impose a general local Lipschitz condition on the nonlinear term $f(u)$ in the fractional operator space $D(A^\theta)$.

\begin{assumption}\label{ass:Lips_L2}
    We postulate the existence of regularity parameters $\gamma_1$ and $\gamma_2$ satisfying $0\leq \gamma_1\leq \gamma_2<1$ and $\gamma_1\leq d/4$, such that the nonlinearity $f(u)$ in \eqref{eq_main} fulfills the local Lipschitz condition
    \begin{equation}\label{Lips_L2}
        \Vert f(u)-f(v)\Vert \leq  \varphi(\Vert A^{\gamma_1}u\Vert+ \Vert A^{\gamma_1}v\Vert) \Vert A^{\gamma_2}(u-v) \Vert, \quad \text{for }u,v\in D(A^{\gamma_2}),
    \end{equation}
    where $\varphi(\cdot)$ is an increasing continuous function.
\end{assumption}

With this assumption, and by combining the analytical techniques from \cite[Theorem 4.1]{ABsPEE}, we deduce the local well-posedness of the mild solution \eqref{eq_mildSolu}.

\begin{theorem}\label{thm:localWelPos}
    Let $u_0 \in D(A^\gamma)$ with $\gamma \in (\gamma_1, d/4]$, and suppose that the nonlinear term $f(u)$  satisfies Assumption \ref{ass:Lips_L2}. Then problem \eqref{eq_mildSolu} admits a unique local solution $u$ such that
    \begin{equation}\label{sol_regular}
        u \in \mathcal{C}([0,T_{u_0}];D(A^\gamma)) \cap \mathcal{C}((0,T_{u_0}];D(A)), 
    \end{equation}
    where $T_{u_0} > 0$ depends only on the norm $\Vert A^{\gamma}u_0\Vert$.
\end{theorem}

\begin{remark}
    Compared with \cite[Theorem 3.1]{AK19}, Theorem \ref{thm:localWelPos}, which relies on Assumption \ref{ass:Lips_L2}, is more realistic in the context of initial data with low regularity, while still yielding the same regularity for the mild solution.
    The parameter $\gamma_1$ represents the minimal regularity of the initial data required by Theorem \ref{thm:localWelPos}, and minimizing this value is often desirable. A concrete instance is provided in \cite[Example 1]{HERK25}, where the nonlinearity $f(u)=u^m\:(m \geq 2)$ yields the specific exponents $\gamma_1=\frac{d(m-2)}{4(m-1)}$ and $\gamma_2=\frac{d}{4}+\varepsilon$.  
\end{remark}

\begin{remark}
    Theorem \ref{thm:localWelPos} only establishes the local well-posedness of mild solutions. 
    To ensure that mild solution $u$ satisfies the original equation \eqref{eq_main}, it requires the additional condition $u \in \mathcal{C}^\alpha([0,T_{u_0}];X)$. Achieving this regularity, however, calls for additional nonlinear assumptions, e.g., Assumption \ref{ass:Lips_Agamma}. Furthermore, the existence of global solutions depends on the specific form of the nonlinearity and is not addressed here. For simplicity, we henceforth assume that equation \eqref{eq_main} admits a global solution. 
\end{remark}

Turning to the numerical analysis of the semilinear problem \eqref{eq_main}, a Lipschitz condition on the nonlinearity is instrumental for establishing stability for the numerical scheme \eqref{NumSch_yn}. 
However, condition \eqref{Lips_L2} is insufficient for an error analysis in the $L^2$ space. 
To overcome this limitation, we employ a duality argument which relaxes the restriction on $\gamma_2$ and leverages enhanced regularity (replacing $\gamma_1$ with $\gamma$). These considerations lead us to posit the following assumption. 

\begin{assumption}\label{ass:Lips_Agamma}
    For $\mu\in [0,\gamma]$, we assume the existence of a regularity parameter $\hat{\gamma}_{\mu} \in (0,1)$ with $\hat{\gamma}_\mu+\mu<1$, such that the nonlinear term $f(u)$ in \eqref{eq_main} satisfies the local Lipschitz condition
    \begin{equation}\label{Lips_Agamma}
        \Vert A^{-\hat{\gamma}_\mu}(f(u)-f(v))\Vert \leq  \varphi(\Vert A^{{\gamma}}u\Vert+ \Vert A^{{\gamma}}v\Vert) \Vert A^{\mu}(u-v)\Vert, \quad \text{for } u,v\in D(A^\gamma), 
    \end{equation}
    where $\varphi(\cdot)$ is an increasing continuous function.
\end{assumption}

\begin{remark}
    In Assumption \ref{ass:Lips_Agamma},  the parameter $\mu$ indicates that the error analysis is carried out in the space $D(A^\mu)$. Upon applying the Lipschitz condition \eqref{Lips_Agamma}, the residual operator $A^{\hat{\gamma}_\mu}$ can be absorbed by the smoothing properties of the solution operators $F(t)$ and $E(t)$ (see Lemma \ref{le:smoo_solOpe}).
\end{remark}

By utilizing the solution regularity \eqref{sol_regular} in conjunction with the Lipschitz condition \eqref{Lips_Agamma} for the case $\mu=\gamma$, we derive the uniform bound
\begin{equation}\label{est_f_Agamma}
    \Vert A^{-\hat{\gamma}_\gamma}f(u(t))\Vert\leq \Vert A^{-\hat{\gamma}_\gamma}(f(u(t))-f(0))\Vert+\Vert  A^{-\hat{\gamma}_\gamma}f(0) \Vert \leq  C, \quad \hat{\gamma}_\gamma+\gamma<1,\: 0\leq t\leq T. 
\end{equation}
To facilitate the estimation of higher-order temporal derivatives of the nonlinear term---crucial for analyzing the local error of the numerical scheme \eqref{NumSch_yn} --- we invoke the chain rule. For any integer $m\geq 1$, this yields
\begin{equation}\label{chainRule}
\begin{aligned}
    \frac{\text{d}^m}{\text{d} t^m}f(u(t)) &= C_1 f'(u(t))  \partial_t^m u(t)
    \\
    & \quad + \sum_{i=2}^{m}\:\:  \sum_{\substack{s_1+\cdots+s_i=m\\ s_i\in \mathbb{N}^+}}  C_i f^{(i)}(u(t)) \partial_t^{s_1} u(t)...\partial_t^{s_i} u(t).
\end{aligned}
\end{equation}
Motivated by the analysis in \cite[Assumption 4]{HERK25}, we introduce the following simplifying hypothesis to control the growth of these derivatives. 

\begin{assumption}\label{ass:FrechetDiff}
    We assume that the nonlinear term $f(u)$ in \eqref{eq_main} is $k$-times Fréchet differentiable with respect to $u$.  
    For $1\leq m\leq k$, if the solution of equation \eqref{eq_main} satisfies the estimate
    \begin{equation*}
        \Vert A^s \partial^{\ell}_t u(t)\Vert \leq C t^{-\ell+\alpha (\gamma-s)}+C,\quad 0\leq s\leq \gamma,\: 1\leq\ell\leq m-1, 
    \end{equation*}
    then we postulate the existence of a regularity parameter $\gamma_4$ with $\gamma+\gamma_4<1$, such that the first term on the right-hand side of \eqref{chainRule} dominates; i.e.,
    \begin{equation}
        \left\Vert A^{-\gamma_4}\frac{\text{d}^m}{\text{d} t^m}f(u(t))\right\Vert \leq C\Vert \partial^m_t u(t)\Vert + C t^{-m+\alpha\gamma}.
    \end{equation}
\end{assumption}

\begin{remark}
    Assumption \ref{ass:FrechetDiff} pertains to the structural properties of the Fréchet derivatives of the nonlinear function $f(\cdot)$ and plays a pivotal role in the inductive proof of Proposition \ref{pro:solRegu}. It is justified by examples such as polynomial nonlinearities, as discussed in \cite[Example 4]{HERK25}. 
\end{remark}

Equipped with the aforementioned hypotheses, we proceed to establish fundamental estimates for the derivatives of both the solution $u(t)$ and the nonlinear term $f(u(t))$. 
Unlike the bound $\|\partial_t^k u(t)\|_{L^\infty}\le C t^{-m}$ in \cite[Theorem 3.2]{ECQ22}, the following estimate \eqref{est_Dtf} is obtained in the weaker space $D(A^{-\gamma_4})$ and exhibits a milder singularity. 
\begin{proposition}\label{pro:solRegu}
    Let the nonlinear term $f(u)$ satisfy Assumptions \ref{ass:Lips_L2}--\ref{ass:FrechetDiff}. If $u(t)$ is the solution of problem \eqref{eq_main}, then $u(t) \in \mathcal{C}^k((0, T]; X)$ and the following estimates hold: 
    \begin{align}
        \Vert A ^\theta \partial^{m}_t u&(t)\Vert \leq C t^{-m+\alpha (\gamma-\theta)}+C, \quad  0\leq \theta\leq d/4,\:1\leq m \leq k, \label{est_Athe_Dtu}
        \\
        &\left\Vert A^{-\gamma_4}\frac{\text{d}^m}{\text{d} t^m}f(u(t))\right\Vert \leq C t^{-m+\alpha\gamma}, \quad 1\leq m \leq k. \label{est_Dtf}
    \end{align}
\end{proposition}

Building upon these preparatory results, we now state the primary contribution of this paper.
\begin{theorem}\label{thm:convergence}
    Let $u_0\in D(A^{\gamma})$ and $\mu\in [0,\gamma]$ with $\gamma\in (\gamma_1,d/4]$. Let $u(t)$ and $u_n$ be the solution of problem \eqref{eq_main} and the discrete scheme \eqref{NumSch_un}--\eqref{NumSch_yn} at time $t_n$, respectively. If the nonlinear term $f(u)$ satisfies Assumptions \ref{ass:Lips_L2}--\ref{ass:FrechetDiff} and the parameter $\nu$ in the stepsize choice \eqref{def_gradeMesh} satisfies $1-(1+\alpha \gamma)/
    k < \nu < 1$,     
    then there exist $\tau^*>0$ and $M^*>0$ such that for $\tau<\tau^*$ and $M>M^*$, we have 
    \begin{equation}\label{res_conve}
        \Vert u(t_n)-u_n\Vert_\mu\leq C \tau^{k} t_n^{-\eta}+ C e^{-\sqrt{M}/C},
    \end{equation}
    where $\eta=\alpha(\gamma_4+\mu-1)+k-\alpha\gamma-k\nu>0$.
\end{theorem}

\begin{remark}
    The error estimate \eqref{res_conve} admits several potential extensions.
    First, refined choices of the auxiliary parameters in the proof lead to alternative forms of the estimate; see Remark \ref{rem:extension} for details. 
    Second, the estimate in the $H^1$ norm can be obtained by taking $\mu=1/2$. 
    When $\gamma\geq 1/2$, this follows directly from Theorem \ref{thm:convergence}; when $\gamma<1/2$, the argument can be extended to cover this choice of $\mu$, following the analogous approach in \cite[Theorem 3]{ERK_FNM}. 
    This aspect is related to the work \cite{HS20}, which established $H^1$-norm error estimates for the L1 scheme, albeit under the stronger regularity assumption $u_0\in H^4\cap H^1_0$.
\end{remark}

\section{Proof of the main results}\label{sec:proof_res}
In this section, we prove the primary convergence theorem. The error analysis is divided into two parts: first, a regularity analysis of the mild solution, and second, the error estimation itself. These two parts are presented in the following two subsections, respectively. 

\subsection{Regularity of the solution}
This subsection investigates the positive-time regularity of the nonlinear problem \eqref{eq_mildSolu}. Our starting point is an examination of the smoothing properties inherent to the solution operators $F(t)$ and $E(t)$.

Observing that $A(z^\alpha + A)^{-1} = I - z^\alpha(z^\alpha + A)^{-1}$, we decompose the operator $F(t)$ as follows:
\begin{equation}\label{def_tilF}
\begin{aligned}
    F(t)u_0
    &= u_0 - \frac{1}{2\pi i}\int_{\Gamma_{\lambda}} e^{zt}z^{-1} A^{1-\gamma}(z^\alpha+A)^{-1} A^\gamma u_0 \text{d}z 
    \\
    &=: u_0- \widetilde{F}(t)  A^\gamma u_0. 
\end{aligned}
\end{equation}
The operator $\widetilde{F}(t)$ satisfies $A^\gamma \widetilde{F}(t) = \int_0^t A E(s) \text{d} s$. Moreover, by operator interpolation theory we also have the resolvent estimate
\begin{equation}\label{resol_Atheta}
    \Vert A^{\theta} (z^\alpha+A)^{-1} \Vert\leq |z|^{-\alpha(1-\theta)}, \quad 0\leq \theta \leq 1.
\end{equation}

Combining the results from \cite[Theorem 1.6]{JZ23} and \cite[Lemma 3.4(i)]{JL18} with operator interpolation theory yields the following estimates.

\begin{lemma} \label{le:smoo_solOpe}
    For $\ell \in \mathbb{N}$ and $t\in(0,T]$, the operators $F(t)$, $E(t)$, and $\widetilde{F}(t)$ satisfy the following bounds:
    \begin{align}
        \quad \quad  \Vert A^\theta \widetilde{F}(t)\Vert&\leq C,    &&0\leq \theta \leq \gamma,\quad \quad & \label{est_A_tilF}
        \\
        \Vert A^{\theta} F^{(\ell)}(t)\Vert&\leq C t^{-\ell-\alpha\theta},    &&0\leq \theta \leq 1,& \label{est_AFl}
        \\
        \Vert A^\theta E^{(\ell)}(t)\Vert &\leq C t^{-1-\ell+\alpha(1-\theta)},\quad  &&0\leq \theta \leq 2, &\label{est_AEl}
        \\
        \Vert A^\theta \widetilde{F}^{(\ell+1)}(t)\Vert  &\leq C t^{-\ell-1+\alpha (\gamma-\theta)},  &&0\leq \theta \leq 1. &\label{est_A_tilFl}
    \end{align}
\end{lemma}

We now proceed to demonstrate the existence of a local solution for problem \eqref{eq_mildSolu}, which follows the proof strategies of \cite[Theorem 4.1]{ABsPEE}. 

\begin{proof}[Proof of Theorem \ref{thm:localWelPos}]
    For any $S \in (0, T]$, we introduce the Banach space
    $$
        \mathcal{X}(S) = \left\{ v \in \mathcal{C}((0, S]; {D}(A^{\gamma_2})) \cap \mathcal{C}([0, S]; {D}(A^{\gamma_1})); \sup_{0 < t \leq S} t^{({\gamma_2} - {\gamma_1})\alpha} \| A^{\gamma_2} v(t) \| < \infty \right\},$$
    equipped with the norm
    $$
        \| v \|_{\mathcal{X}(S)} = \sup_{0 < t \leq S} t^{({\gamma_2} - {\gamma_1})\alpha} \| A^{\gamma_2} v(t) \| + \sup_{0 \leq t \leq S} \| A^{\gamma_1} v(t) \|.$$
    Within $\mathcal{X}(S)$ we consider a bounded subset $\mathcal{K}(S)$ whose radius $R>0$ will be chosen later. Clearly, $\mathcal{K}(S)$ is a nonempty closed subset of $\mathcal{X}(S)$. 
    For $v\in\mathcal{K}(S)$, we define the mapping
    $$
        \{\Phi v\}(t)=F(t) u_0+\int_0^t E(t-s)f(v(s))\text{d}s,\quad 0\leq t\leq S.$$

    \textbf{Step 1.} We first show that $\Phi$ maps $\mathcal{K}(S)$ into itself. For $v \in \mathcal{K}(S)$, the Lipschitz condition \eqref{Lips_L2} implies
    \begin{equation*}
        \|f(v(t))\| \leq C_{R}(t^{({\gamma_1} - {\gamma_2})\alpha} + 1), \quad 0 < t \leq S.
    \end{equation*}   
    For any $\theta$ satisfying ${\gamma_1} \leq \theta < 1$, utilizing estimates \eqref{est_AFl} and \eqref{est_AEl} yields
    \begin{equation}\label{Atheta_Phi}
    \begin{aligned}
        t^{(\theta - {\gamma_1})\alpha} \|A^\theta \{\Phi v\}(t)\| 
        & \leq  C_{\theta} \Vert A^{\gamma_1} u_0 \Vert + t^{(\theta - {\gamma_1})\alpha} \int_0^t (t - s)^{-1+\alpha(1-\theta)} C_{R}(s^{({\gamma_1} - {\gamma_2})\alpha} + 1) \text{d}s 
        \\
        &\leq  C_{\theta} \Vert A^{\gamma_1} u_0 \Vert + C_{\theta, R} t^{(1-{\gamma_2})\alpha} + C_{\theta, R} t^{(1-{\gamma_1})\alpha}.
    \end{aligned}
    \end{equation}
    Choosing $\theta = {\gamma_1}$ and $\theta = {\gamma_2}$ in \eqref{Atheta_Phi}, we observe that if $R$ is taken such that
    $$
        R > (C_{{\gamma_1}}+C_{{\gamma_2}}) \|A^{\gamma_1} u_0\|, $$
    and if $S$ is sufficiently small, then $\Vert \Phi(v)\Vert_{\mathcal{X}(S)}<R $.  
    Together with the continuity of the operators $E(t)$ and $F(t)$, this implies that $\Phi$ maps $\mathcal{K}(S)$ into itself.

    \textbf{Step 2.} Next we verify that $\Phi$ is a contraction on $\mathcal{X}(S)$. Let $v_1, v_2 \in \mathcal{K}(S)$ be arbitrary. From Lipschitz condition \eqref{Lips_L2} and estimate \eqref{est_AEl}, we deduce
    \begin{align*}
        t^{(\theta-{\gamma_1})\alpha} \Vert A^{\theta}[\{\Phi v_1\}(t)-\{\Phi v_2\}(t)]\Vert & \leq  C_{\theta,R} t^{(\theta-{\gamma_1})\alpha}\int_{0}^{t}(t-s)^{-1+\alpha(1-\theta)} \Vert A^{\gamma_2}[v_1(s)-v_2(s)]\Vert \text{d}s 
        \\
        & \leq C_{\theta,R} S^{\alpha(1-{\gamma_2})} \Vert v_1-v_2\Vert_{\mathcal{X}(S)}, \quad {\gamma_1} \leq \theta < 1.
    \end{align*}
    Evaluating these estimates at $\theta={\gamma_1}$ and $\theta={\gamma_2}$ gives
    $$
        \|\Phi v_1-\Phi v_2\|_{\mathcal{X}(S)}\leq C_{\theta,R} S^{\alpha(1-{\gamma_2})} \Vert v_1-v_2\Vert_{\mathcal{X}(S)},$$
    demonstrating that $\Phi$ is a contraction provided $S>0$ is sufficiently small. The contraction mapping theorem thus guarantees a unique fixed point, corresponding to the unique solution of problem \eqref{eq_mildSolu}. 

    \textbf{Step 3.} We now improve the regularity of the obtained solution. From the arguments in Steps 1 and 2, the local existence time $T_{u_0} = S > 0$ depends solely on $\Vert A^{\gamma_1} u_0\Vert$.  Adapting the techniques from \cite[Theorem 4.2]{ABsPEE} alongside Lemma \ref{le:smoo_solOpe}, we derive
    $$  
        \Vert A^{\gamma_2} u(t)\Vert \leq C(t^{\alpha(\gamma-{\gamma_2})}+ 1), \quad  0 < t \leq T_{u_0}.$$
    which immediately leads to the estimate for the nonlinearity:
    \begin{equation}\label{est_f_L2}
        \Vert f(u(t))\Vert\leq C(t^{\alpha(\gamma-\gamma_2)}+1).
    \end{equation}
    Moreover, a direct computation shows
    $$  
        \Vert A^\gamma u(t)\Vert \leq C + \int_0^t (t-s)^{-1+\alpha(1-\gamma)} (s^{\alpha(\gamma-{\gamma_2})}+1) \leq C, $$
    implying $u \in \mathcal{C}([0,T_{u_0}];D(A^\gamma))$. 
    
    To further upgrade the regularity, let $h>0$, $t>0$, and $\theta\in[0,\gamma_2]$. Using \eqref{def_tilF}, Lemma \ref{le:smoo_solOpe}, and \eqref{est_f_L2}, we analyze the difference:
    \begin{align*}
        \Vert A^{{\theta}}(u(t+h)-u(t))\Vert & \leq \Vert A^{\theta} ( \widetilde{F}(t+h)-\widetilde{F}(t)) A^{\gamma}u_0 \Vert + \int_t^{t+h} A^{\theta} E(s)f(u(t+h-s)) \text{d}s
        \\
        &\quad +\int_0^t A^{{\theta}}E(s)(f(u(t+h-s))-f(u(t-s))) \text{d}s 
        \\
        & \leq C \int_t^{t+h} \Vert A^{{\theta}}\widetilde{F}'(s)\Vert \text{d}s  + \int_t^{t+h} s^{-1+\alpha(1-{\theta})}((t+h-s)^{\alpha(\gamma-{\gamma_2})}+1)\text{d}s
        \\
        &\quad + \int_0^{t}  s^{-1+\alpha(1-{\theta})} \Vert A^{\gamma_2} (u(t+h-s)-u(t-s)) \Vert \text{d}s
        \\
        & \leq C t^{-1+\alpha \gamma}h^{1-\alpha\theta} + C t^{-1+\alpha ({1-\theta})} (h^{1+\alpha(\gamma-{\gamma_2})} +h)
        \\
        & \quad + \int_0^{t}  (t-s)^{-1+\alpha(1-{\theta})} \Vert A^{\gamma_2} (u(s+h)-u(s)) \Vert \text{d}s.
    \end{align*}
    Taking $\theta=\gamma_2$ and applying Gronwall's inequality yields
    \begin{equation}\label{est_Agam_uDiff}
        \Vert A^{{\gamma_2}}(u(t+h)-u(t))\Vert \leq  C t^{-1+\alpha\gamma}h^{1-\alpha \gamma_2} + C t^{-1+\alpha ({1-\gamma_2})} (h^{1+\alpha(\gamma-{\gamma_2})}+h).
    \end{equation}
    Applying the operator $A$ to both sides of \eqref{eq_mildSolu}, employing the identity $A^\gamma \widetilde{F}(t) = \int_0^t A E(s) \text{d} s$, and using Lemma \ref{le:smoo_solOpe}, \eqref{est_f_L2}, and \eqref{est_Agam_uDiff}, we obtain
    \begin{align*}
        \Vert A u(t)\Vert &\leq \Vert A F(t)u_0\Vert + \Vert A^\gamma \widetilde{F}(t) f(u(t))\Vert  + \int_0^t A E(t-s) (f(u(s))-f(u(t)))\text{d}s 
        \\
        &\leq Ct^{-\alpha(1-\gamma)}\Vert A^\gamma u_0\Vert + C(t^{-\alpha(\gamma_2-\gamma)}+1) + \int_0^t (t-s)^{-1} \Vert A^{\gamma_2} (u(s)-u(t))\Vert\text{d}s 
        \\
        &\leq Ct^{-\alpha(1-\gamma)}, 
    \end{align*}
    which establishes the regularity $u\in\mathcal{C}((0,T_{u_0}];{D}(A))$. 
\end{proof}

We next derive estimates for the higher-order derivatives of the solution $u(t)$ and the nonlinear term $f(u(t))$, which are indispensable for the subsequent error analysis. 

\begin{proof}[Proof of Proposition \ref{pro:solRegu}]   
    Following \cite[Theorem 3.2]{ECQ22}, for $m\geq 1$ we have
    \begin{equation}\label{def_t_ell_u}
        t^{m-1} u(t)= t^{m-1} {F}(t)u_0+\sum_{j=0}^{m-1} \begin{pmatrix}m-1\\j\end{pmatrix}w_{m-1,j}(t)
    \end{equation}
    with 
    \begin{align}
        w_{m-1,j}(t) &= \int_0^t (t - s)^j E(t - s) s^{m-1-j} f(u(s)) \text{d}s,
        \\
        \partial_{t}^{m-1}w_{m-1,j}(t)&=\int_{0}^{t}\partial_{s}^{j}(s^{j}E(s))\frac{d^{m-1-j}}{dt^{m-1-j}}[(t-s)^{m-1-j}f(u(t-s))]\text{d}s.\label{def_Dt_omega}
    \end{align}
    For the case $j=m-1$, differentiating \eqref{def_Dt_omega} once more and using Assumption \ref{ass:FrechetDiff}, Lemma \ref{le:smoo_solOpe}, and \eqref{est_f_Agamma} yields
    \begin{equation}\label{est_w_ell_1}
    \begin{aligned}
        \Vert A^\theta \partial_{t}^{m}w_{m-1,m-1}(t)\Vert &= \left\Vert  A^\theta \partial_{t} \int_{0}^{t}\partial_{s}^{m-1}(s^{m-1}E(s)) f(u(t-s))\text{d}s \right\Vert
        \\
        &\leq \Vert A^{\hat{\gamma}_\gamma+\theta}\partial_{t}^{m-1}(t^{m-1}E(t)) A^{-\hat{\gamma}_\gamma}f(u_0)\Vert 
        \\
        &\quad+ \int_{0}^{t} \Vert A^{\gamma_4+ \theta} \partial_{s}^{m-1}(s^{m-1}E(s)) A^{-\gamma_4} \frac{d}{dt}f(u(t-s))\Vert \text{d}s
        \\
        &\leq C t^{-1+\alpha(1-\hat{\gamma}_\gamma-\theta)}+ \int_0^t (t-s)^{-1+\alpha(1-\gamma_4)} \Vert \partial_s u(s) \Vert\text{d}s. 
    \end{aligned}
    \end{equation}

    Differentiating \eqref{def_t_ell_u} with $m=1$ and substituting the above estimate, we obtain
    \begin{align*}
        \Vert  A^\theta \partial_t  u(t)\Vert 
        &\leq \Vert  \widetilde{F}'(t) A^\gamma u_0\Vert + \Vert A^\theta \partial_t w_{0,0}(t)\Vert 
        \\
        &\leq  C t^{-1+\alpha (\gamma-\theta)} + C t^{-1+\alpha(1-\hat{\gamma}_\gamma-\theta)}+ C\int_0^t (t-s)^{-1+\alpha(1-\gamma_4)} \Vert \partial_s u(s) \Vert\text{d}s.
        \\
        &\leq  C t^{-1+\alpha (\gamma-\theta)} + C\int_0^t (t-s)^{-1+\alpha(1-\gamma_4)} \Vert A^{\theta} \partial_s u(s) \Vert\text{d}s.
    \end{align*}
    An application of Gronwall's inequality then verifies \eqref{est_Athe_Dtu} for $m=1$. 
    We now proceed by induction, assuming \eqref{est_Athe_Dtu} holds for all $m = 2,\dots, \ell-1$ to prove it for $m = \ell$. 

    We first estimate $\Vert \partial_{t}^{\ell}w_{\ell-1,j}(t)\Vert$ for the remaining cases. For $0<j < \ell-1$, applying Assumption \ref{ass:FrechetDiff} and Lemma \ref{le:smoo_solOpe} gives
    \begin{equation}\label{est_w_ell_2}
    \begin{aligned}
        \Vert A^\theta \partial_{t}^{\ell}w_{\ell-1,j}(t)\Vert &\leq \int_{0}^{t}\Vert A^{\gamma_4+\theta} \partial_{s}^{j}(s^{j}E(s))\Vert \left\Vert A^{-\gamma_4} \frac{d^{\ell-j}}{dt^{\ell-j}}[(t-s)^{\ell-1-j}f(u(t-s))]\right\Vert \text{d}s
        \\
        &\leq \int_0^t s^{-1+\alpha(1-\gamma_4-\theta)} (t-s)^{-1+\alpha\gamma} \text{d}s
        \\
        &\leq Ct^{-1+\alpha(\gamma-\theta)}.
    \end{aligned}
    \end{equation}
    Similarly, for $j=0$,
    \begin{equation}\label{est_w_ell_3}
    \begin{aligned}
        \Vert A^\theta\partial_{t}^{\ell}w_{\ell-1,0}(t)\Vert &\leq \int_{0}^{t}\Vert A^{\gamma_4+\theta} E(s)\Vert \left\Vert A^{-\gamma_4} \frac{d^{\ell}}{dt^{\ell}}[(t-s)^{\ell-1}f(u(t-s))]\right\Vert \text{d}s
        \\
        &\leq Ct^{-1+\alpha(\gamma-\theta)} + \int_{0}^{t} (t-s)^{-1+\alpha(1-\gamma_4-\theta)} \Vert A^\theta s^{\ell-1} \partial_t^{\ell}  u(s)\Vert \text{d}s,
    \end{aligned}
    \end{equation}
    where a change of variables was utilized in the final step. 
    By taking the $l$-th derivative of \eqref{def_t_ell_u} and inserting the estimates \eqref{est_w_ell_1}--\eqref{est_w_ell_3} into it, we obtain
    \begin{equation}\label{est_A_D_tu}
    \begin{aligned}
        \Vert A^\theta \partial_t^\ell(t^{(\ell-1)} u(t))\Vert& \leq\Vert A^\theta \partial_t^\ell( t^\ell\widetilde{F}(t) A^\gamma u_0)\Vert+\sum_{j=0}^\ell
        \begin{pmatrix}\ell\\j\end{pmatrix}
        \Vert A^\theta \partial_t^\ell w_{\ell,j}(t)\Vert
        \\
        &\leq Ct^{-1+\alpha(\gamma-\theta)} + \int_{0}^{t} (t-s)^{-1+\alpha(1-\gamma_4-\theta)} \Vert A^\theta s^{\ell-1} \partial_t^{\ell}  u(s)\Vert \text{d}s.
    \end{aligned}
    \end{equation}
 
    By the product rule we have
    \begin{align*}
        \Vert A^\theta t^{\ell-1}\partial_{t}^{\ell}u(t)\Vert & \leq \Vert A^\theta \partial_{t}^{\ell}(t^{\ell-1}u(t))\Vert  + C\sum_{j=1}^{\ell-1}\Vert A^\theta t^{\ell-j-1}\partial_{t}^{\ell-j}u(t)\Vert  \\
        &\leq \Vert A^\theta \partial_{t}^{\ell-1}(t^{\ell}u(t))\Vert  + C t^{-1+\alpha(\gamma-\theta)},
    \end{align*}
    where the induction hypothesis for \eqref{est_Athe_Dtu} was employed in the last inequality. Combining this with \eqref{est_A_D_tu} implies
    \begin{equation}\label{estAut_Gronw}
        \| A^\theta t^{\ell-1}\partial_{t}^{\ell}u(t)\| \leq Ct^{-1+\alpha(\gamma-\theta)} + \int_{0}^{t}C(t-s)^{-1+\alpha(1-\gamma_4-\theta)}\|A^\theta s^{\ell-1}\partial_{t}^{\ell}u(s)\| \text{d}s.
    \end{equation}
    A final application of Gronwall's inequality establishes \eqref{est_Athe_Dtu}. Consequently, estimate \eqref{est_Dtf} follows directly from Assumption \ref{ass:FrechetDiff}.
\end{proof}

\begin{remark}
    The proof above follows the main steps of \cite[Theorem 3.2]{ECQ22}, with only minor modifications. 
    In particular, the first term on the right-hand side of \eqref{estAut_Gronw} gains an extra factor $\alpha(\gamma-\theta)$ due to the exploitation of the initial regularity, leading to a milder singularity. 
    To accommodate this factor, the exponent in the definition \eqref{def_t_ell_u} is reduced from $t^m$ to $t^{m-1}$.
\end{remark}

\subsection{Error estimates}
We commence the error estimate by establishing the quadrature error arising from the approximation in \eqref{eq_quadAppro}. For the first integral component, the integrand exhibits double-exponential decay. Consequently, invoking \cite[Theorem 2.1]{LF10}, the quadrature converges exponentially with respect to $M$: 
\begin{equation}\label{err_q1}
    \|\mathcal{E}_{1,q}(t)\|_\mu \leq Ce^{-C M/\ln M}\Vert u_0\Vert_\mu \quad \text{for } 0\leq \mu\leq \gamma.
\end{equation}

For the second component, instead of directly estimating the Laplace transform of $f(u(t))$, we follow the technique from \cite[Lemma~3.3]{ECQ22} by incorporating the factor $e^{zt}$.
This adjustment reduces the integrand to a single-exponential decay, which is then handled by the following lemma.

\begin{lemma}\label{le:err_q2}
    For $\mu\in [0,\gamma]$, suppose there exist parameters $\theta_1$ and $\theta_2$ with $0\leq\theta_1< 1-\mu$ and $0\leq \theta_2<1$ such that 
    \begin{equation}\label{ass_f_bound}
        \Vert A^{-\theta_1} f(u(t))\Vert \leq C t^{-\theta_2}, \quad 0<t\leq T.
    \end{equation}
    Then the quadrature error satisfies
    $$
    \|\mathcal{E}_{2,q}(t)\|_{\mu}\leq C t^{-\theta_2}e^{-\sqrt{C \alpha(1-\mu-\theta_1) M}}.$$
\end{lemma}

\begin{proof}
    Under hypothesis \eqref{ass_f_bound}, we estimate $y(z,t)$:
    $$\begin{aligned}
        \Vert A^{-\theta_1} y(z,t)\Vert &\leq\int_0^te^{\text{Re}(z)(t-s)}\Vert A^{-\theta_1} f(u(s))\Vert \text{d}s\leq C\int_0^t e^{\text{Re}(z)(t-s)}s^{-\theta_2} \text{d}s
        \\
        &\leq \int_0^{t/2} e^{\text{Re}(z)(t-s)} s^{-\theta_2} \text{d}s + \int_{t/2}^t e^{\text{Re}(z)(t-s)}s^{-\theta_2} \text{d}s
        \\
        &\leq C t^{-\theta_2} |z|^{-1} \quad \text{for } z\in \Gamma_{\tilde{\lambda}}.
        \end{aligned}$$   
    Combining this with the resolvent bound \eqref{resol_Atheta} yields
    \begin{equation}\label{est_Fz}
        \Vert G(z)\Vert_\mu := \Vert A^{\mu+\theta_1}(z^\alpha+A)^{-1} A^{-\theta_1} y(z,t)\Vert \leq Ct^{-\theta_2}|z|^{-1-\alpha(1-\mu-\theta_1)}\quad \text{for }z\in \Gamma_{\tilde{\lambda}}.
    \end{equation}
    We define $g(s)=G(\tilde{\zeta}(s))\tilde{\zeta'}(s)$ where the hyperbolic contour $\tilde{\zeta}(s)$ is specified in \eqref{def_contour}. Writing $s=x+iy$, we have $\tilde{\zeta}(s)=\tilde{\lambda}(1-\sin(\beta-y)\cosh(x))-i\tilde{\lambda}\cos(\beta-y)\sinh(x)$. It is straightforward to verify that $g(x+iy)$ is analytic in the strip
    $$
        D=\{s\in\mathbb{C}:|\mathrm{Im}(s)|\leq \tilde{d}\}\quad\text{with any fixed }\tilde{d}\in\left(0,\frac\pi2-\beta\right).$$
    Since $|\tilde{\zeta}'(s)| \simeq |\tilde{\zeta}(s)|$ as $|\operatorname{Re}(s)| \rightarrow \infty$, the bound in \eqref{est_Fz} implies that $g(x+iy)$ exhibits single-exponential decay as $x\to\pm\infty$: 
    $$  
        \Vert A^\mu g(x+iy)\Vert \leq C t^{-\theta_2}\big|\tilde{\zeta}(s)\big|^{-\alpha(1-\mu-\theta_1)} \leq  Ct^{-\theta_2} e^{-\alpha(1-\mu-\theta_1)|x|} \quad\text{for } x\in\mathbb{R} \text{ and } |y|\leq \tilde{d}.$$ 
    By \cite[Lemma 3.4]{Thom10}, we finally obtain
    $$
        \|\mathcal{E}_{2,q}(t)\|_\mu \leq C t^{-\theta_2}e^{- \sqrt{2 \pi \tilde{d} \alpha(1-\mu-\theta_1) M}}.$$
    This completes the proof. 
\end{proof}

We are now positioned to prove the main convergence theorem. It follows the strategy of \cite[Theorem 2.1]{ECQ22} and splits the total error into two parts, the main difference being that we additionally have to 
handle the factor $A^{\theta}$, i.e., the singularity transferred from the nonlinear term. 

\begin{proof}[Proof of Theorem \ref{thm:convergence}]
    Let $u^{(\tau)}(s)$ denote the numerical solution, which we take to be piecewise linear in time.  
    Denote by $\hat{I}_\tau f (u^{(\tau)})$ the extrapolation of $f(u)$ using the values of the numerical solution, defined analogously to \eqref{def_extrap}. Then we introduce an intermediate solution $u^*_n$ by
    \begin{align}
        y_{n}(z)&=e^{z\tau_{n}}y_{n-1}(z)+\int_{t_{n-1}}^{t_{n}}e^{z(t_{n}-s)}\hat{I}_{\tau}f(u^{(\tau)})(s)\text{d}s\quad\text{for }z\in\Gamma_{\tilde{\lambda}},
        \\
        u_{n}^{*}&=\frac{1}{2\pi i}\int_{\Gamma_{\lambda}}e^{zt}z^{\alpha-1}(z^{\alpha}+A)^{-1}u_{0}dz+\frac{1}{2\pi i}\int_{\Gamma_{\tilde{\lambda}}}(z^{\alpha}+A)^{-1}y_{n}(z)dz. \label{Numsche_u_star}
    \end{align}
    The total error is decomposed as $e_n = \tilde{e}_n + e_n^*$ with $\tilde{e}_n = u(t_n) - u_n^*$ and $e_n^* = u_n^* - u_n$.

    To analyze the first part $\tilde{e}_n$, we observe that $u^*_n$ coincides with the solution $u^*(t_n)$ of the semilinear subdiffusion equation
    $$\begin{cases}
        \partial_{t}^{\alpha}u^{*}(t)+A u^{*}(t)=\hat{I}_{\tau}f(u^{(\tau)})(t), \quad 0<t\leq T,
        \\
        u^{*}(0)=u_{0}.\end{cases}$$
    Define the interpolation error $\mathcal{E}_f(t) = f(u(t))-\hat{I}_\tau f(u)(t)$. Using the mild solution representation \eqref{eq_mildSolu} for the subdiffusion equation, the difference $\tilde{e}=u-u^*$ can be bounded by
    \begin{equation}\label{err_til_spli}
    \begin{aligned}
        \Vert \tilde{e}(t_n)\Vert_\mu 
        &\leq \left\Vert \int_0^{t_n}A^{\hat{\gamma}_\mu+\mu} E(t_n-s)A^{-\hat{\gamma}_\mu}(\hat{I}_\tau f(u)-\hat{I}_\tau f(u^{(\tau)}))\text{d}s\right\Vert
        \\
        & \quad +\left\Vert \sum_{j=1}^{n} \int_{t_{j-1}}^{t_j}A^{\gamma_4+\mu}E(t_n-s)A^{-\gamma_4}\mathcal{E}_f(s)\text{d}s \right\Vert.
    \end{aligned}
    \end{equation}

    We first estimate the remainder $\Vert A^{-\gamma_4}\mathcal{E}_f(t)\Vert$ for $t\in(t_{j-1},t_j)$, which appears in the second term on the right-hand side. Standard polynomial interpolation theory provides
    \begin{equation}\label{remainder_f_1}
        \Vert A^{-\gamma_4}\mathcal{E}_f(t)\Vert \leq
        \begin{cases}
            C\tau_j^k\max_{s\in[t_{j-k},t_j]}\left\| A^{-\gamma_4}\frac{\text{d}^k}{\text{d} t^k}f(u(s))\right\| &\text{for } j\geq k+1,
            \\
            \int_{t_{j-1}}^{t_j} \left\| A^{-\gamma_4}\frac{\text{d}}{\text{d} t} f(u(s))\right\| \text{d}s &\text{for } 1\leq j< k. 
        \end{cases}
    \end{equation}
    For the case of $j=k$ and $t\in (t_{k-1},t_k]$, expanding $\mathcal{E}_f(t)$ to first order at $t_0$ and exploiting the property $\sum_{i=1}^{k} L_{ki}(s)=1$ yields
    \begin{equation}\label{remainder_f_2}
    \begin{aligned}
        \Vert A^{-\gamma_4}\mathcal{E}_f(t)\Vert \leq 
        \left\|  A^{-\gamma_4} \left( \int_{0}^{t}\frac{\text{d}}{\text{d} t} f(u(s))\text{d}s - \dots - L_{k,k-1}\int_{0}^{t_1}\frac{\text{d}}{\text{d} t} f(u(s))\text{d}s \right)    \right\|.
    \end{aligned}
    \end{equation}
    Merging estimates \eqref{remainder_f_1} and \eqref{remainder_f_2}, and using Proposition \ref{pro:solRegu} together with the graded mesh property $\tau_{j} \simeq \tau_{j+k}$, we derive a uniform bound
    \begin{equation}\label{est_remainder_f}
        \Vert A^{-\gamma_4}\mathcal{E}_f(t)\Vert \leq C \tau_j^k t_j^{-k+\alpha\gamma} \quad \text{for } t\in (t_{j-1},t_j],\: 1\leq j\leq n.  
    \end{equation}

    The first term on the right-hand side of \eqref{err_til_spli} relies on Assumption \ref{ass:Lips_Agamma}, which necessitates the boundedness of $u^{(\tau)}$ in $D(A^\gamma)$. This boundedness will in turn be deduced from the convergence result \eqref{res_conve}. Therefore we proceed by mathematical induction.  
    By \eqref{sol_regular}, there exists $R > 0$ such that $\Vert A^\gamma u(t) \Vert\leq R$ for $t\in [0,T]$. For the base case $n = 0$, we have $\Vert u_0\Vert_\gamma \leq M + 1$. Then assume inductively that
    \begin{equation}\label{ass_induc}
        \Vert u_n\Vert_\gamma\leq M + 1, \quad \text{for } n\leq m-1.  
    \end{equation}
    We aim to prove this bound for $n = m$ provided $\tau$ is sufficiently small and $M$ is sufficiently large. Since
    $$  
        \Vert u_m\Vert_\gamma \leq \Vert u(t_m)-u_m\Vert_\gamma+\Vert u(t_m)\Vert_\gamma,$$
    it suffices to establish the convergence of $\Vert e_m \Vert_\gamma$. 

    Under the induction hypothesis and Assumption \ref{ass:Lips_Agamma}, we have 
    $$
        \Vert A^{-\hat{\gamma}_\mu}(f(u(t_j))-f(u_j))\Vert \leq  C \Vert e_j\Vert_\mu, \quad \text{for } j<m.  $$
    Inserting this estimate and \eqref{est_remainder_f} into \eqref{err_til_spli}, and using Lemma \ref{le:smoo_solOpe} together with the graded mesh property \eqref{def_gradeMesh}, yields
    \begin{equation}\label{est_err1}
    \begin{aligned}
        \Vert \tilde{e}(t_n)\Vert_\mu 
        &\leq C\sum_{j=1}^n\int_{t_{j-1}}^{t_j}(t_n-s)^{\alpha(1-\hat{\gamma}_\mu-\mu)-1}\Vert e_{j}\Vert_\mu \text{d}s + \sum_{j=1}^{n}\int_{t_{j-1}}^{t_j}  (t_n-s)^{\alpha(1-\gamma_4-\mu)-1} \tau_j^k t_j^{\alpha\gamma-k } \text{d}s
        \\
        & \leq  C\sum_{j=1}^n \tau_j(t_n- t_{j-1})^{\alpha(1-\hat{\gamma}_\mu-\mu)-1} \Vert e_{j}\Vert_\mu 
        + C \tau^{k} t_n^{-\eta},
    \end{aligned}
    \end{equation}
    where $k\nu-k+\alpha\gamma>-1$,  $\eta=\alpha(\gamma_4+\mu-1)+k-\alpha\gamma-k\nu>0$,  and $n\leq m$.

    Next we analyze the second part $e^*_n = u^*_n - u_ n$. Consider the difference between \eqref{NumSch_un} and \eqref{Numsche_u_star}:
    \begin{align*}
        A^\mu e_{n}^{*}&=\frac{1}{2\pi i}\int_{\Gamma_{\lambda}}e^{zt}z^{\alpha-1}(z^{\alpha}+A)^{-1} A^\mu u_{0}dz-\sum_{j=-M}^{M}w_{j}z_{j}^{\alpha-1}(t_{n})(z_{j}^{\alpha}(t_{n})+A)^{-1} A^\mu u_{0}
        \\
        &\quad +\frac1{2\pi i}\int_{\Gamma_\lambda}A^{\hat{\gamma}_\gamma+\mu}(z^\alpha+A)^{-1} A^{-\hat{\gamma}_\gamma}y_n(z)\mathrm{d}z-\sum_{j=-M}^{M}\tilde{w}_j A^{\hat{\gamma}_\gamma+\mu}(\tilde{z}_j^\alpha+A)^{-1}A^{-\hat{\gamma}_\gamma}y_n(\tilde{z}_j)
        \\
        &= A^\mu \mathcal{E}_{1,q}(t_n)+ A^\mu\widetilde{\mathcal{E}}_{2,q}(t_n).
    \end{align*}
    By the induction hypothesis \eqref{ass_induc} and Assumption \ref{ass:Lips_Agamma}, $\Vert A^{-\hat{\gamma}_\gamma}f(u_n)\Vert $ is bounded for $1 \leq n \leq m - 1$; consequently,
    $$
        \Vert  A^{-\hat{\gamma}_\gamma} \hat{I}_{\tau} f(u^{(\tau)})(t)\Vert \leq C, \quad 0 \leq t \leq t_m.$$
    Hence, by \eqref{err_q1} and Lemma \ref{le:err_q2} with $\theta_1=\hat{\gamma}_\gamma$ and $\theta_2=0$, we obtain
    \begin{equation}\label{est_err2}
        \Vert e^*_n \Vert_\mu \leq \Vert  \mathcal{E}_{1,q}(t_n)\Vert_\mu + \Vert \tilde{\mathcal{E}}_{2,q}(t_n)\Vert_\mu \leq Ce^{-C M/\ln M} + C e^{- \sqrt{C \alpha(1-\mu-\hat{\gamma}_\gamma) M}}.
    \end{equation}
    Combining \eqref{est_err1} and \eqref{est_err2} via the triangle inequality gives
    \begin{align*}
        \Vert e_n\Vert_\mu &\leq C \tau^{k} t_n^{-\eta}+ Ce^{-C M/\ln M} + C e^{- \sqrt{C \alpha(1-\mu-\hat{\gamma}_\gamma) M}} 
        \\
        &\quad +  C\sum_{j=1}^n \tau_j(t_n- t_{j-1})^{\alpha(1-\hat{\gamma}_\mu-\mu)-1} \Vert e_{j}\Vert_\mu \quad \text{for } n\leq m.
    \end{align*}
    Applying the discrete Gronwall inequality from \cite[Lemma B.1]{ECQ22} yields
    \begin{align}\label{res_conve2}
        \Vert e_n\Vert_\mu &\leq C \tau^{k} t_n^{-\eta}+ Ce^{-C M/\ln M} + C e^{- \sqrt{C \alpha(1-\mu-\hat{\gamma}_\gamma) M}} \quad \text{for } n\leq m.
    \end{align}
    Taking $\mu=\gamma$ and $n=m$, we note that $k-\eta>0$, which guarantees the uniform convergence of $\Vert e_m\Vert_\gamma$ and therefore the numerical solution remains bounded in $D(A^{\gamma})$. This completes the proof. 
\end{proof}

\begin{remark}\label{rem:extension}
The proof of Theorem \ref{thm:convergence} admits refined parameter choices to obtain variants of the error estimate in \eqref{res_conve2}.
(\rmnum{1}). The singularity exponent $\eta$ in the first term is governed by the grading parameter $\nu$. Increasing $\nu$ can enforce $\eta = 0$, yielding uniform $k$-order convergence. 
(\rmnum{2}). The quadrature error in the third term is derived via Lemma \ref{le:err_q2} with $\theta_1 = \hat{\gamma}_\gamma, \theta_2 = 0$. 
In practice, increasing $\theta_2$ reduces $\theta_1$, thereby improving the exponential decay rate of the quadrature error away from $t_0$, albeit introducing singularities near the initial time. 
Collectively, these adjustments redistribute the singularities between the solution operator (resolvent) and the nonlinearity within the estimation procedure.
\end{remark}

\section{Numerical experiments} \label{sec:num_exper}

\begin{table}[b]
	\centering
	\caption{Temporal discretization errors $\Vert U_{\tau}-U_{\tau/2}\Vert_\alpha\:(\alpha=0,1)$ and convergence orders of the $2$-step method for the model with nonlinear term (a) and initial values (\rmnum{1})--(\rmnum{3}), $k=2$.  Top: $L^2$ norm errors; bottom: $H^1$ norm errors.}
    \label{tab:temConv_k2NL1}
    \begin{adjustbox}{max width=\textwidth}
	\begin{tabular}{cccccccccc}
		\hline
		\multirow{2}{*}{$\tau$} & & \multicolumn{2}{c}{(\rmnum{1}), $\nu = 0.5$} &&  \multicolumn{2}{c}{ (\rmnum{2}), $\nu = 0.45$}  && \multicolumn{2}{c}{(\rmnum{3}), $\nu = 0.4$}
		\\ \cline{3-4} \cline{6-7} \cline{9-10} 
		~& & Error & Order &&  Error & Order &&  Error & Order 
		\\ \hline
        $1/2^6$ && 2.701E-03 & -- & ~ & 5.039E-03 & -- & ~ & 1.495E-02 & --  \\ 
        $1/2^7$ && 6.695E-04 & 2.01 & ~ & 1.181E-03 & 2.09 & ~ & 3.856E-03 & 1.96  \\ 
        $1/2^8$ && 1.606E-04 & 2.06 & ~ & 2.770E-04 & 2.09 & ~ & 9.878E-04 & 1.96  \\ 
        $1/2^9$ && 3.759E-05 & 2.09 & ~ & 6.538E-05 & 2.08 & ~ & 2.498E-04 & 1.98  \\    
        \hline
        $1/2^6$ && 1.529E-02 & -- & ~ & 2.820E-02 & -- & ~ & 8.183E-02 & --  \\ 
        $1/2^7$ && 3.790E-03 & 2.01 & ~ & 6.609E-03 & 2.09 & ~ & 2.110E-02 & 1.96  \\ 
        $1/2^8$ && 9.092E-04 & 2.06 & ~ & 1.551E-03 & 2.09 & ~ & 5.406E-03 & 1.96  \\ 
        $1/2^9$ && 2.129E-04 & 2.09 & ~ & 3.661E-04 & 2.08 & ~ & 1.367E-03 & 1.98 \\ 
        \hline
	\end{tabular}
    \end{adjustbox}
\end{table}

This section reports several numerical experiments carried out to confirm the theoretical analysis. 
We investigate the semilinear subdiffusion model \eqref{eq_main} with $\alpha=0.5$ on the unit square $\Omega = (0,1)\times (0,1)$ up to $T = 1$.
The sectorial operator $A$ is taken as the realization of the Laplacian $-\Delta$ in $L^2(\Omega)$ subject to homogeneous Dirichlet boundary conditions.
For the spatial discretization we employ the piecewise linear Galerkin finite element method on a mesh that is sufficiently fine, so that it does not affect the observation of the time discretization errors.
The time stepping is performed via the exponential convolution quadrature (ECQ) method \eqref{NumSch_un}--\eqref{NumSch_yn} proposed in \cite{ECQ22}.
All computations are carried out in MATLAB R2025a on a personal laptop.

Our main objective is to determine the convergence order of the ECQ method at $T$ for the semilinear subdiffusion problem \eqref{eq_main} under different initial data and nonlinear terms. 
The following three initial values are examined.
\begin{enumerate}[label=(\roman*).]
    \item $u_0(x_1,x_2) = ((x_1-0.5)^2+(x_2-0.5)^2)^{-0.49}-0.5^{-0.49} \in L^2$ but $u_0 \notin  D(A^{0.01})$.  
    \item $u_0(x_1,x_2) = ((x_1-0.5)^2+(x_2-0.5)^2)^{-0.24}-0.5^{-0.24} \in D(A^{0.25})$ but $u_0 \notin  D(A^{0.26})$. 
    \item $u_0(x_1,x_2) =  5(x_1(1-x_1))(x_2(1-x_2)) \in D(A^{0.5})$.
\end{enumerate}
Two nonlinearities are considered:
\begin{enumerate}[label=(\alph*)]
    \item $f(u)=5u(10-u)$ with $\gamma_1=0$. 
    \item $f(u)= u(10-u)(15-u)$ with  $\gamma_1 = 0.25$. 
\end{enumerate}
According to the examples discussed in \cite{HERK25}, the above nonlinear terms $f(u)$ satisfy Assumptions \ref{ass:Lips_L2}--\ref{ass:FrechetDiff}.

\begin{table}[h]
	\centering
	\caption{Temporal discretization errors $\Vert U_{\tau}-U_{\tau/2}\Vert_\alpha\:(\alpha=0,1)$ and convergence orders of the $3$-step method for the model with nonlinear term (a) and initial values (\rmnum{1})--(\rmnum{3}).  Top: $L^2$ norm errors; bottom: $H^1$ norm errors.}
    \label{tab:temConv_k3NL1}
    \begin{adjustbox}{max width=\textwidth}
	\begin{tabular}{cccccccccc}
		\hline
		\multirow{2}{*}{$\tau$} & & \multicolumn{2}{c}{ (\rmnum{1}), $\nu = 0.67$} &&  \multicolumn{2}{c}{ (\rmnum{2}), $\nu = 0.63$}  && \multicolumn{2}{c}{ (\rmnum{3}), $\nu = 0.59$}
		\\ \cline{3-4} \cline{6-7} \cline{9-10} 
		~& & Error & Order &&  Error & Order &&  Error & Order 
		\\ \hline
        $1/2^5$ && 5.949E-04 & -- & ~ & 1.457E-03 & -- & ~ & 4.963E-03 & --  \\ 
        $1/2^6$ && 7.158E-05 & 3.05 & ~ & 1.593E-04 & 3.19 & ~ & 6.110E-04 & 3.02  \\ 
        $1/2^7$ && 8.207E-06 & 3.12 & ~ & 1.780E-05 & 3.16 & ~ & 7.313E-05 & 3.06  \\ 
        $1/2^8$ && 9.526E-07 & 3.11 & ~ & 2.094E-06 & 3.09 & ~ & 8.629E-06 & 3.08  \\    
        \hline
        $1/2^5$ && 3.341E-03 & -- & ~ & 8.178E-03 & -- & ~ & 2.713E-02 & --  \\ 
        $1/2^6$ && 4.022E-04 & 3.05 & ~ & 8.948E-04 & 3.19 & ~ & 3.341E-03 & 3.02  \\ 
        $1/2^7$ && 4.611E-05 & 3.12 & ~ & 1.001E-04 & 3.16 & ~ & 3.999E-04 & 3.06  \\ 
        $1/2^8$ && 5.357E-06 & 3.11 & ~ & 1.178E-05 & 3.09 & ~ & 4.718E-05 & 3.08 \\ 
        \hline
	\end{tabular}
    \end{adjustbox}
\end{table}

\begin{table}[h]
	\centering
	\caption{Temporal discretization errors $\Vert U_{\tau}-U_{\tau/2}\Vert_\alpha\:(\alpha=0,1)$ and convergence orders for the model with nonlinear term (b) and initial values (\rmnum{2})--(\rmnum{3}), obtained by the $2$-step (left) and $3$-step (right) methods. Top: $L^2$ norm errors; bottom: $H^1$ norm errors.} 
    \label{tab:temConv_k23NL2}
	\begin{adjustbox}{max width=16cm}
	\begin{tabular}{cccccccc:cccccc}
		\hline
		\multirow{2}{*}{$\tau$} & & \multicolumn{2}{c}{ (\rmnum{2}), $\nu = 0.45$} && \multicolumn{2}{c}{ (\rmnum{3}), $\nu = 0.4$} &&& \multicolumn{2}{c}{ (\rmnum{2}), $\nu = 0.63$} && \multicolumn{2}{c}{(\rmnum{3}), $\nu = 0.59$}
		\\ \cline{3-4} \cline{6-7} \cline{10-11} \cline{13-14}
		~& & Error & Order && Error & Order &&& Error & Order && Error & Order
		\\ \hline
		$1/2^5$ && 1.535E-02 & -- & ~ & 8.598E-04 & -- & ~ && 1.171E-04 & -- & ~ & 4.319E-04 & -- \\
		$1/2^6$ && 3.523E-03 & 2.12 & ~ & 2.402E-04 & 1.84 & ~ && 1.260E-05 & 3.22 & ~ & 5.566E-05 & 2.96 \\
		$1/2^7$ && 8.038E-04 & 2.13 & ~ & 6.364E-05 & 1.92 & ~ && 1.384E-06 & 3.19 & ~ & 6.758E-06 & 3.04 \\
		$1/2^8$ && 1.841E-04 & 2.13 & ~ & 1.640E-05 & 1.96 & ~ && 1.616E-07 & 3.10 & ~ & 7.991E-07 & 3.08 \\
		\hline
		$1/2^5$ && 8.573E-02 & -- & ~ & 4.685E-03 & -- & ~ && 7.923E-04 & -- & ~ & 3.434E-03 & -- \\
		$1/2^6$ && 1.968E-02 & 2.12 & ~ & 1.309E-03 & 1.84 & ~ && 8.350E-05 & 3.25 & ~ & 4.446E-04 & 2.95 \\
		$1/2^7$ && 4.494E-03 & 2.13 & ~ & 3.468E-04 & 1.92 & ~ && 8.998E-06 & 3.21 & ~ & 5.398E-05 & 3.04 \\
		$1/2^8$ && 1.030E-03 & 2.12 & ~ & 8.935E-05 & 1.96 & ~ && 1.037E-06 & 3.12 & ~ & 6.372E-06 & 3.08 \\
		\hline
	\end{tabular}
	\end{adjustbox}
\end{table}

Theorem \ref{thm:convergence} states that, by choosing suitable parameters $M$ and $\nu$, the $k$-step method achieves convergence of order $k$ in the $L^2$ norm.
The choice of $M$ should guarantee that $e^{-\sqrt{M}/C}$ is smaller than $t_n^{-\eta}\tau^k$.
Hence we set $M = 10 |\ln(1/\tau)|^2$, which ensures that the temporal convergence order is not masked by this parameter.
Regarding the step sizes in \eqref{def_gradeMesh}, two strategies are adopted in the experiment.
The first one uses a graded mesh $\{t_n \:|\: t_n = (n/N)^r T,\; r = 1/(1-\nu) \}$ which is directly generated and satisfies the mesh assumption \eqref{def_gradeMesh} with $\tau = rT/N$.
This mesh is employed for the relatively stable $2$-step method.
For the $3$-step method, the initial step sizes can be adjusted to enhance stability; we set $\tau_1=\tau_2=\tau_3$ and generate the remaining steps recursively by $\tau_n = (t_{n-1}/T)^{\nu}\tau$ for $n\geq 4$. 
The total numbers of steps produced by these two approaches are comparable.

Because an exact solution is not available for the problem under study, the numerical convergence order is estimated by
\begin{equation*}
    \log\left(\frac{\Vert U_{\tau}-U_{\tau/2} \Vert } {\Vert U_{\tau/2}-U_{\tau/4} \Vert}\right)/\log(2),
\end{equation*} 
using the three finest meshes, where $U_{\tau}$ denotes the numerical approximation at $T$ obtained with the maximum step size $\tau$.

The numerical errors for the problem \eqref{eq_main} with nonlinearity (a) are displayed in Tables \ref{tab:temConv_k2NL1}--\ref{tab:temConv_k3NL1}.
Compared with the results in \cite{ECQ22}, our work admits initial data that are not in $L^\infty$, such as (\rmnum{1}) and (\rmnum{2}), and the requirement on the graded mesh parameter $\nu$ is slightly relaxed to $\nu > 1-\frac{1+\alpha \gamma}{k}$, featuring the additional factor $\frac{\alpha \gamma}{k}$.
From Tables \ref{tab:temConv_k2NL1}--\ref{tab:temConv_k3NL1} we observe that the $k$-step method indeed exhibits $k$th-order convergence in the $L^2$ norm for initial values (\rmnum{1})--(\rmnum{3}).
We also present results in the $H^1$ norm, where $k$th-order convergence is still observed. This suggests that our conclusions can be further extended, similarly to the discussion in \cite[Theorem 3]{ERK_FNM}.

For nonlinearity (b), the parameter $\gamma_1$ is larger than that of nonlinearity (a), which restricts the admissible range of initial values. We therefore carry out the same experiments only for initial values (\rmnum{2}) and (\rmnum{3}); the corresponding results are reported in Table \ref{tab:temConv_k23NL2} and are again consistent with the expected behavior.

\bibliographystyle{elsarticle-num}
\bibliography{biob}

%

\end{document}